\documentclass[a4paper, 11pt]{article}
\usepackage{amsmath,amssymb,amsthm,xypic} 
\input xy
\xyoption{all}

\usepackage[hypertex]{hyperref}

\usepackage[totalheight=24 true cm, totalwidth=15 true cm]{geometry}
\usepackage{color}

\title{An intrinsic characterization of Picard-Vessiot extensions}
\author{Michael Wibmer\\ RWTH Aachen \\michael.wibmer@matha.rwth-aachen.de}

\newtheorem{theo}{Theorem}[section]
\newtheorem{lemma}[theo]{Lemma}

\newtheorem{cor}[theo]{Corollary}
\newtheorem{defi}[theo]{Definition}

\theoremstyle{definition}

\newcommand{\Gl}{\operatorname{GL}}

\newcommand{\D}{\Delta}

\newcommand{\de}{\delta}

\begin{document}

\maketitle

\begin{abstract}
We present a characterization of differential Picard-Vessiot extensions which is analogous to the field theoretic characterization of Galois extensions in classical Galois theory.
\end{abstract}

\section{Introduction}
Roughly speaking there are two approaches to classical Galois theory: There is the explicit equational approach where one starts with a polynomial, constructs the splitting field and realizes the Galois group as a permutation group of the roots. On the other hand, there is also the intrinsic field theoretic approach where the Galois group is a group of field automorphisms and the Galois extension is characterized by a normality property: A separable algebraic field extension $L|K$ is Galois if and only if for every field extension $M$ of $L$ and every $K$-morphism $\tau\colon L\to M$ one has $\tau(M)=M$. (Of course it suffices to take as $M$ an algebraic closure of $K$ containing $L$.)

In the Galois theory of linear differential equations the explicit equational approach is predominating. However, in \cite[Prop. 3.9, p. 27]{Magid:LecturesOnDifferentialGaloisTheory} (see also \cite{Magid:DifferentialGaloisTheoryNotices}) one can find a characterization of Picard-Vessiot extensions which is ``getting rid'' of the equation.

\begin{theo}
Let $L|K$ be a differential field extension \footnote{In this introduction all differential fields are understood to be ordinary and of characteristic zero. Moreover the constants of the base differential field $K$ are assumed to be algebraically closed.} and $k$ the field of constants of $K$. Then $L|K$ is Picard-Vessiot if and only if
\begin{enumerate}
 \item $L$ is generated over $K$ by a finite dimensional $k$-vector space $V\subset L$.
\item There is a group $G$ of differential automorphisms of $L|K$ with $G(V)\subset V$ and $L^G=K$.
\item $L|K$ has no new constants.
\end{enumerate}
\end{theo}

Another intrinsic characterization is given in \cite[Def. 1.8, p. 132 and Theorem 3.11, p. 141]{AmanoMasuokaTakeuchi:HopfPVtheory}.

\begin{theo}
Let $L|K$ be a finitely differentially generated differential field extension with no new constants. Then $L|K$ is Picard-Vessiot if and only if there exists a differential $K$-subalgebra $R$ of $L$ with quotient field $L$ such that $R\otimes_K R$ is generated by constant elements as a left $R$-module.
\end{theo}

None of the above two characterizations is really close in spirit to the field theoretic characterization of classical Galois extension by normality. On the other hand, Kolchin's definition of strongly normal extensions can be seen as a differential analog of the classical normality condition. Recall that a finitely differentially generated extension of differential fields $L|K$ without new constants is strongly normal if for every differential field extension $M|L$ and every differential embedding $\tau\colon L\to M$ over $K$ one has
$\tau(L)\subset LM^\de$. Here $M^\de$ are the constants of $M$ and $LM^\de$ is the field compositum of $L$ and $M^\de$ inside $M$.

Owing to the theory of strongly normal extensions one can give another intrinsic characterization of Picard-Vessiot extensions
(\cite[Theorem 2, p. 891]{Kolchin:OntheGaloisTheoryofDifferentialfields} or \cite[Theorem 10.5, p. 4149]{Kovacic:GeometricCharacterizationofStronglyNormal}): An extension of differential fields is Picard-Vessiot if and only if it is strongly normal and the Galois group is linear (i.e., affine).

This characterization presupposes knowledge of the Galois group as an algebraic group and is not entirely field theoretic. So a characterization of Picard-Vessiot extensions by ``normality'' is still missing. Here we will fill in this gap: One can strengthen the strongly normal property by replacing the differential field extension $M$ of $L$ by an arbitrary differential $L$-algebra. Our main result is that this strengthening precisely yields a characterization of Picard-Vessiot extensions:

\begin{theo}
Let $L|K$ be a finitely differentially generated differential field extension with no new constants. Then $L|K$ is Picard-Vessiot if and only if for every differential $L$-algebra $M$ and every differential embedding $\tau\colon L\to M$ over $K$ we have $\tau(L)\subset LM^\de$.
\end{theo}

We will also establish an analogous result for the case of several commuting derivations. Our proof is based on J. Kovacic's scheme-theoretic approach to the Galois theory of strongly normal extensions which avoids universal differential fields and Kolchin's idiosyncratic axiomatic algebraic groups.

\section{Picard-Vessiot extensions and normality}

All rings are assumed to be commutative and all fields are assumed to be of characteristic zero. As it adds little difficulty we shall also treat the partial case and not just the ordinary case as in the introduction. So let $\D=\{\de_1,\ldots,\de_m\}$ denote a finite set of commuting derivations. We shall use the prefix $\D$ in lieu of the word ``differential''.
We will mostly stick to the conventions of our basic references \cite{Kovacic:differentialgaloistheoryofstronglynormal} and \cite{Kovacic:GeometricCharacterizationofStronglyNormal}. Throughout we assume that the constants $k=K^\D$ of our base differential field $K$ are algebraically closed.

If $S$ and $T$ are subrings of a ring $R$ we denote with $S\cdot T$ the smallest subring of $R$ that contains $S$ and $T$. So explicitly
$S\cdot T=\left\{\sum_i s_it_i| \ s_i\in S,\ t_i\in T\right\}$. With $ST$ we denote the smallest subring of $R$ that contains $S$ and $T$ and is closed under taking inverses, so explicitly
$ST=\{\frac{r}{r'}|\ r,r'\in S\cdot T,\ r'\in R^\times\}$. If $S,T$ and $R$ are fields then $ST$ is simply the field compositum.

\begin{defi}
Let $L|K$ be an extension of $\D$-fields and $\tau_s,\tau_t\colon L\to R$ a pair of $K$-$\D$-morphisms into some $K$-$\D$-algebra $R$. We say that $L|K$ is \emph{$\D$-normal with respect to $(\tau_s,\tau_t)$} if $\tau_t(L)\subset \tau_s(L)R^\D$. If $L|K$ is $\D$-normal with respect to any such pair for every $K$-$\D$-algebra $R$ we say that $L|K$ is \emph{$\D$-normal}.
\end{defi}

\begin{defi}
Let $L|K$ be a finitely $\D$-generated extension of $\D$-fields such that $L^\D=K^\D$. Then $L|K$ is \emph{strongly normal} if $L|K$ is $\D$-normal with respect to any pair $\tau_s,\tau_t\colon L\to M$ of $K$-$\D$-morphisms into some \emph{$\D$-field extension} $M$ of $K$.
\end{defi}
The above definition is slightly different from Kolchin's original one (\cite[p. 393]{Kolchin:differentialalgebraandalgebraicgroups}) who of course always has a universal $\D$-field lurking in the background. The definitions are easily seen to be equivalent (\cite[Prop. 12.2, p. 4489]{Kovacic:differentialgaloistheoryofstronglynormal}).
Moreover one can show (\cite[Prop. 12.4, p. 4490]{Kovacic:differentialgaloistheoryofstronglynormal}) that a strongly normal $\D$-extension is finitely generated as field extension.
We also recall the definition of Picard-Vessiot extensions.

\begin{defi}
Consider a system of linear differential equations of the form
\begin{equation} \label{eq:linear equation}
\delta_i(y)=A_iy, \ i=1,\ldots,m
\end{equation}
over the base $\D$-field $K$, where the matrices $A_i\in K^{n\times n}$ satisfy the integrability conditions
\[\de_i(A_j)-\de_j(A_i)=A_iA_j-A_jA_i, \ 1\leq i,j\leq m.\]

A $\D$-field extension $L$ of $K$ with $L^\D=K^\D$ is called a \emph{Picard-Vessiot extension} for the system (\ref{eq:linear equation}) if it is generated by a fundamental solution matrix, i.e., there exists $Y\in\Gl_n(L)$ such that $\de_i(Y)=A_iY$ for $i=1,\ldots,m$ and $L=K(Y)$.
\end{defi}
By a Picard-Vessiot extension we mean a differential field extension which is a Picard-Vessiot extension for some integrable linear system.
We note that in the ordinary case ($m=1$) the integrability conditions are vacuous and one obtains the usual characterization of Picard-Vessiot extensions.

We now state some preparatory results which will be needed in the proof of the main theorem.

\begin{lemma} \label{lemma: intermediate}
Let $L|K$ be a $\D$-normal extension of $\D$-fields. Then $L|M$ is $\D$-normal for every intermediate $\D$-field $K\subset M\subset L$.
\end{lemma}
\begin{proof}
This is clear because a $\D$-$M$-morphism is a $\D$-$K$-morphism.
\end{proof}

\begin{lemma} \label{lemma: PV transitive}
Let $L|K$ be a strongly normal extension of $\D$-fields and let $K^\circ$ denote the relative algebraic closure of $K$ inside $L$. Assume that $L|K^\circ$ is Picard-Vessiot. Then $L|K$ is Picard-Vessiot.
\end{lemma}
\begin{proof}
The Galois group of $L|K^\circ$ is the identity component of the Galois group of $L|K$ (\cite[Prop. 31.3, p. 4515]{Kovacic:differentialgaloistheoryofstronglynormal}). If the identity component of an algebraic is affine then the algebraic group itself must be affine. Thus the claim of the lemma follows from the fact that the strongly normal extensions with affine Galois group are precisely the Picard-Vessiot extensions (\cite[Theorem 10.5, p. 4149]{Kovacic:GeometricCharacterizationofStronglyNormal}).
\end{proof}

One of the main ingredients of our proof is the following units theorem due to Ax, Lichtenbaum, Halperin and Sweedler. See \cite[Cor. 1.8, p. 264]{Sweedler:unitsTheorem}.

\begin{theo}
Let $R$ and $S$ be algebras over the field $K$ such that $K$ is absolutely algebraically closed in $R$ and $S$. Then every invertible element of $R\otimes_K S$ is of the form $a\otimes b$ with $a$ invertible in $R$ and $b$ invertible in $S$.
\end{theo}

Here $K$ is called absolutely algebraically closed in $R$ if $M$ is integrally closed in $R\otimes_K M$ for every field extension $M$ of $K$.

\begin{cor} \label{cor:units}
Let $L|K$ be an extension of fields such that $K$ is relatively algebraically closed in $L$. Then every invertible element $u$ of $L\otimes_K L$ is of the form $u=a\otimes b$ for some $a,b\in L\smallsetminus\{0\}$.
\end{cor}
\begin{proof}
We only have to see that $K$ is absolutely algebraically closed in $L$. Let $M$ be a field extension of $K$. Since $L$ is regular over $K$ (remember, we are assuming characteristic zero throughout) it follows from \cite[Prop. 8, Chapter V, \S 17, A.V.142]{Bourbaki:Algebra2} that the field of fractions of the integral domain $L\otimes_K M$ is a regular extension of $M$. In particular $M$ is integrally closed in $L\otimes_K M$.
\end{proof}

Now we are prepared to prove the main theorem:

\begin{theo} \label{theo:main}
Let $L|K$ be a finitely $\D$-generated extension of $\D$-fields with $L^\D=K^\D$. Then $L|K$ is Picard-Vessiot if and only if $L|K$ is $\D$-normal.
\end{theo}
\begin{proof}
The ``only if''-statement is easy: Assume that $L$ is a Picard-Vessiot extension of $K$ for an integrable system
\begin{equation} \label{eq:linear equation 2}
\delta_i(y)=A_iy, \ i=1,\ldots,m
\end{equation}
with $A_i\in K^{n\times n}$ ($i=1,\ldots,m$) and fundamental solution matrix $Y\in\Gl_n(L)$.
Let $R$ be a $K$-$\D$-algebra and $\tau_s,\tau_t\colon L\to R$ a pair of $K$-$\D$-morphisms. Then $\tau_s(Y)$ and $\tau_t(Y)$ are fundamental solution matrices for (\ref{eq:linear equation 2}) inside $\Gl_n(R)$. A well-known computation (cf. \cite[p. 8]{SingerPut:differential}) shows that there exists a matrix $C\in\Gl_n(R^\D)$ such that $\tau_t(Y)=\tau_s(Y)C$. Because $L$ is generated by the entries of $Y$ this shows that $\tau_t(L)\subset \tau_s(L)R^\D$. Therefore $L|K$ is $\D$-normal.

\bigskip

Now assume that $L|K$ is $\D$-normal. We have to show that $L|K$ is Picard-Vessiot. We will first reduce to the case that $K$ is relatively algebraically closed in $L$: Let $K^\circ$ denote the relative algebraic closure of $K$ inside $L$. By Lemma \ref{lemma: intermediate} the extension $L|K^\circ$ is $\D$-normal. So, by assumption, $L|K^\circ$ is Picard-Vessiot and it follows from Lemma \ref{lemma: PV transitive} that $L|K$ is Picard-Vessiot.

So from now we will assume that $K$ is relatively algebraically closed in $L$. Clearly $L|K$ is strongly normal, so all the results form \cite{Kovacic:differentialgaloistheoryofstronglynormal} and \cite{Kovacic:GeometricCharacterizationofStronglyNormal} are available.
We consider $L\otimes_K L$ as $L$-algebra via the first factor.

We set $D:=(L\otimes_K L)^\D$ and $R:=\{a\in L|\ 1\otimes a\in L\cdot D\}\subset L$. Then $R$ is a $K$-$\D$-subalgebra of $L$ and an element of $R$ is said to be Picard-Vessiot over $K$ (\cite[Def. 8.5, p. 4147]{Kovacic:GeometricCharacterizationofStronglyNormal}). One knows (\cite[Prop. 10.1, p. 4148]{Kovacic:GeometricCharacterizationofStronglyNormal}) that the quotient field of $R$ is a Picard-Vessiot extension of $K$. Thus it will suffice to show that the quotient field of $R$ is equal to $L$.


So let $b\in L$. By assumption $L|K$ is $\D$-normal with respect to $\tau_s,\tau_t\colon L\to L\otimes_K L$ where $\tau_s(a)=a\otimes 1$ and $\tau_t(a)=1\otimes a$ respectively. Consequently $1\otimes b\in LD$. This means that we can express $1\otimes b$ in the form
\begin{equation} \label{eq: for b}
1\otimes b=\frac{\sum a_id_i}{u}
\end{equation}
where $a_i\in L$, $d_i\in D$ and $u\in L\cdot D$ is a unit of $L\otimes_K L$. As $K$ is relatively algebraically closed in $L$ we can apply the units theorem (Corollary \ref{cor:units}) to $L\otimes_K L$ to conclude that $u$ is of the form $u=a'\otimes a$ with $a',a\in L\smallsetminus\{0\}$. As $u\in L\cdot D$ we see that $1\otimes a\in L\cdot D$, i.e., $a\in R$.
Multiplying equation (\ref{eq: for b}) with $1\otimes a$ we find that
\[1\otimes ab=\sum\tfrac{a_i}{a'}d_i\in L\cdot D.\]
Therefore also $ab\in R$ and $b=\frac{ab}{a}$ lies in the quotient field of $R$, as desired.

\end{proof}

The idea of replacing the field by an algebra in the normality condition does have an analog in classical Galois theory. In fact, it is well known that a finite field extension $L|K$ is Galois if and only if $1\otimes L\subset L\otimes 1\cdot E(L\otimes_KL)$. Here the role of the constants is played by the idempotent elements $E(L\otimes_K L)$ of $L\otimes_K L$. The reader is referred to \cite{Borceuxetal:GaloisTheories} for a presentation of Galois theory that emphasizes the role of idempotent elements.

We note that for a finite field extension $L|K$ every non-zero divisor of $L\otimes_K L$ is invertible. So taking the total ring of fractions would not make a difference. However, in the differential world it makes a difference. Taking the ring of fractions leads to the strongly normal theory whereas not taking the ring of fractions leads to Picard-Vessiot theory:

\begin{theo}
Let $L|K$ be a finitely $\D$-generated extension of $\D$-fields with $L^\D=K^\D$.
\begin{enumerate}
\item $L|K$ is Picard-Vessiot if and only if $L|K$ is $\D$-normal with respect to $\tau_s,\tau_t\colon L\to L\otimes_K L$ given by $\tau_s(a)=a\otimes 1$ and $\tau_t(a)=1\otimes a$.
\item $L|K$ is strongly normal if and only if $L|K$ is $\D$-normal with respect to $\tau_s,\tau_t\colon L\to \operatorname{Quot}(L\otimes_K L)$ given by $\tau_s(a)=a\otimes 1$ and $\tau_t(a)=1\otimes a$.
\end{enumerate}
\end{theo}
\begin{proof}
As the first step we will show that $L|K$ is strongly normal if $L|K$ is $\D$-normal with respect to $\tau_s,\tau_t\colon L\to \operatorname{Quot}(L\otimes_K L)$ given by $\tau_s(a)=a\otimes 1$ and $\tau_t(a)=1\otimes a$.

For this we need to know that $L\otimes_K L$ has only finitely many minimal prime ideals. One way to see this is as follows: Assume that $a_1,\ldots,a_n\in L$ generate $L$ as $\D$-field over $K$. Let $R$ denote the $K$-$\D$-subalgebra of $L\otimes_K L$ generated by $a_1\otimes 1,\ldots,a_n\otimes 1, 1\otimes a_1,\ldots,1\otimes a_n$. Because we are in characteristic zero $L\otimes_K L$ is reduced and so also $R$ is reduced. It thus follows from the Ritt-Raudenbusch basis theorem that $R$ has only a finite number of minimal prime ideals. As $L\otimes_K L$ is a localization of $R$ also $L\otimes_K L$ has finitely many minimal prime ideals. So $L\otimes_K L$ is reduced and has only a finite number of minimal prime ideals. This implies that $M:=\operatorname{Quot}(L\otimes_K L)$ is a finite direct product of fields. So we may write $M=e_1M\oplus\cdots\oplus e_rM$ where $e_1,\ldots,e_r\in M$ are pairwise orthogonal idempotent elements such that $e_iM$ is a field for $i=1,\ldots,r$. From $e_i^2=e_i$ we deduce that $2e_i\de_j(e_i)=\de_j(e_i)$ so that $\de_j(e_i)\in e_iM$ for $i=1,\ldots,r$ and $j=1,\ldots,m$.
It follows that the $e_iM$'s are $\D$-fields and that $e_1,\ldots,e_r\in M^\D$.

For every $i=1,\ldots,r$ we obtain a pair of $K$-$\D$-morphisms $\tau_{s,i}\tau_{t,i}\colon L\to e_iM$ by composing $a\mapsto a\otimes 1$ and $a\mapsto 1\otimes a$ with the projection $M\to e_iM$. As $L|K$ is $\D$-normal with respect to $\tau_s,\tau_t\colon L\to \operatorname{Quot}(L\otimes_K L)$ given by $\tau_s(a)=a\otimes 1$ and $\tau_t(a)=1\otimes a$ we infer that
$L|K$ is $\D$-normal with respect to ($\tau_{s,i},\tau_{t,i}$) for $i=1,\ldots,r$. It now follows from \cite[Prop. 10, Chapter VI, Section 3, p. 393]{Kolchin:differentialalgebraandalgebraicgroups} that $L|K$ is strongly normal. (As explained in \cite[Section 7]{Kovacic:differentialgaloistheoryofstronglynormal} (equivalence classes of) ``isomorphisms'' correspond to prime $\D$-ideals of $L\otimes_K L$ and ``specialization of isomorphisms'' corresponds to inclusion of ideals. Moreover the $e_iM$'s are the residue fields of the minimal prime ideals of $L\otimes_K L$.)

Now we are in a position to prove (i). The ``only if'' implication of (i) is clear from Theorem \ref{theo:main}. So we assume that $L|K$ is $\D$-normal with respect to $\tau_s,\tau_t\colon L\to L\otimes_K L$. Clearly this implies that $L|K$ is $\D$-normal with respect to $\tau_s,\tau_t\colon L\to \operatorname{Quot}(L\otimes_K L)$. As seen above this implies that $L|K$ is strongly normal. Let $K^\circ$ denote the relative algebraic closure of $K$ in $L$. By Lemma \ref{lemma: PV transitive} it suffices to show that $L|K^\circ$ is Picard-Vessiot. Considering the natural surjection $L\otimes_K L\to L\otimes_{K^\circ} L$ we find that $L|K^\circ$ is $\D$-normal with respect to $\tau_s,\tau_t\colon L\to L\otimes_{K^ \circ} L$ given by $\tau_s(a)=a\otimes 1$ and $\tau_t(a)=1\otimes a$. Now we can proceed as in the proof of Theorem \ref{theo:main} to deduce that $L|K^\circ$ is Picard-Vessiot.

It remains to prove that $L|K$ is $\D$-normal with respect to $\tau_s,\tau_t\colon L\to \operatorname{Quot}(L\otimes_K L)$ if $L|K$ is strongly normal. Let $a\in L$. Then for every $i=1,\ldots,r$ we can write down a formula expressing the fact that $\tau_{t,i}(a)=e_i(1\otimes a)\in \tau_{s,i}(L)(e_iM)^\D\subset e_iM$. Since the $e_i$'s are constant these formulas can be patched together to yield a formula expressing that $1\otimes a\in \tau_s(L)M^\D$.
\end{proof}

\bibliographystyle{alpha}
\bibliography{bibdata.bib}

\end{document}